\documentclass[a4paper,11pt]{amsart}
\usepackage{latexsym, amssymb,amsmath,amsthm}

\numberwithin{equation}{section}

\newtheorem{theorem}{Theorem}[section]
\newtheorem{lemma}[theorem]{Lemma}
\newtheorem{proposition}[theorem]{Proposition}
\newtheorem{corollary}[theorem]{Corollary}
\newtheorem{definition}[theorem]{Definition}
\newtheorem{example}[theorem]{Example}
\newtheorem{remark}[theorem]{Remark}

\usepackage[colorlinks,unicode=false,linkcolor=blue,citecolor=blue]{hyperref}

\begin{document}

\title{Matrices as graded BiHom-algebras and decompositions}

\author{Jiacheng Sun}
\address{School of Mathematics, Southeast University, Nanjing 211189, China} 
\email{220242018@seu.edu.cn}

\author{Shuanhong Wang}
\address{Shing-Tung Yau Centre, School of Mathematics, Southeast University, Nanjing 210096, China}
\email{shuanhwang@seu.edu.cn}

\author{Haoran Zhu}
\address{School of Physical and Mathematical Sciences, Nanyang Technological University, 21 Nanyang Link, 637371, Singapore}
\email{haoran.horace.zhu@gmail.com}

\date{\today}
\subjclass[2020]{16W50, 16S80, 15B30, 17B70}
\keywords{matrix algebra, graded algebra, decomposition, regular BiHom-associative algebra}

\begin{abstract}
We present matrices as graded regular BiHom-associative algebras and
study several features of their decompositions. More precisely, we introduce a notion of connection in the support of the grading and use it to construct a family of canonical graded ideals. We show that, under suitable assumptions, including $\Sigma$-multiplicativity, maximal length and triviality of the centre, the matrix regular BiHom-associative algebra decomposes into a direct sum of graded simple ideals. We further extend our results to general graded regular BiHom-associative algebras over arbitrary base fields. As applications, we reinterpret classical gradings on matrix algebras such as those induced by Pauli matrices and the $\mathbb{Z}_n \times \mathbb{Z}_n$-grading in terms of our setting.
\end{abstract}

\maketitle

\section{Introduction}

Graded structures and decomposition techniques for associative algebras have been extensively studied in recent years~\cite{CDM,PPS}. A prototypical class is that of graded matrix algebras~\cite{BZ,C2,D,EDN,GG1,GG2,K,TG}. Beyond their intrinsic interest, they provide a rich source of examples of graded Lie algebras, which in turn play notable roles in integrable systems~\cite{SWZ3,V}, general relativity~\cite{CS3}, and electroweak interactions~\cite{CFS}.

This paper studies the structure of the graded matrix regular BiHom-associative algebra $(\mathcal{M}_I(\mathcal{R}),\star)$ in the framework of gradings by a regular BiHom-group $(G,+_{(\alpha,\beta)},\alpha,\beta,\mathbf{0})$. Our approach generalises the matrix-grading constructions developed in~\cite{C2,AragonCalderon2012} to the BiHom setting. Unless otherwise stated, $\mathcal{R}$ is a commutative base ring; when dimensions are used, as in Section~\ref{sect5}, we work over fields.

The origins of Hom-algebra structures can be traced back to the physics literature, where they arose in the context of $q$-deformations of Lie algebras of vector fields. Hartwig, Larsson and Silvestrov~\cite{HLS,LS} introduced Hom-Lie algebras to capture the algebraic structure underlying deformations of the Witt and Virasoro algebras.

Subsequently, various Hom-type structures have been developed, including Rota--Baxter Hom-algebras~\cite{MZ}, BiHom-NS algebras~\cite{LMM}, Hom-incidence structures~\cite{FKS}, Hom-Lie groups~\cite{PNM}, and Sweedler duality for Hom-type algebras~\cite{SWZZ}. One noteworthy point is that, when considering two different commuting structure maps, Graziani, Makhlouf, Menini and Panaite introduced the notion of BiHom-type algebras in~\cite{GMM}. In the special case where the two twisting maps coincide, a BiHom-algebra reduces to a Hom-algebra, and when both twisting maps are the identity map, it recovers an ordinary algebra. This framework has received increasing attention in recent years, leading to the development of related structures such as BiHom-type Lie superalgebras, BiHom-Lie colour algebras and BiHom-associative algebras; see, for instance, \cite{HMN,AHM,GMM,SWZ3}.

The connection technique used in this paper is a version of the method of connections among roots, weights or support elements. This method first appeared in the study of split Lie algebras with symmetric root systems by Calder\'on Mart\'in~\cite{CalderonMartin2008}, and has since become a useful tool in the study of graded, split and related algebraic structures. Its basic idea is to regard two non-zero homogeneous degrees as connected when one can pass from one to the other through a finite chain of non-zero homogeneous products. The corresponding equivalence classes then provide a natural way to construct graded ideals and to describe direct-sum decompositions. Similar ideas have also been used in Hom-type and BiHom-type settings; see, for instance, the recent work on graded regular BiHom-Lie algebras in~\cite{BarreiroCalderonNavarroSanchez2025}.

The paper is organised as follows. In Section~\ref{sect2}, we review some basic preliminaries about graded algebras and regular BiHom-associative algebras. In Section~\ref{sect3}, we introduce the notion of \emph{connection} within the support $\Sigma$ of a grading and prove that this relation defines an equivalence relation. This framework enables us to organise the homogeneous components via equivalence classes. In Section~\ref{sect4}, we apply the connection technique to decompose the graded matrix regular BiHom-associative algebra $(\mathcal{M}_I(\mathcal{R}),\star,\psi,\phi)$ as a direct sum of graded ideals, each indexed by a distinct equivalence class in $\Sigma$. We further show that these ideals are mutually orthogonal, and we establish sufficient conditions under which the decomposition is direct. In Section~\ref{sect5}, we extend the results to general graded regular BiHom-associative algebras over arbitrary base fields. We introduce the notions of $\Sigma$-multiplicativity and maximal length, and show that, under these assumptions, the algebra admits a direct-sum decomposition into graded simple ideals, each supported on a connected subset of $\Sigma$.

Throughout the paper, $\mathbb{N}$ will denote the set of non-negative integers and $\mathbb{Z}$ the set of integers.

\section{Preliminaries}\label{sect2}

In this section, we recall some basic notions concerning BiHom-associative algebras~\cite{GMM} and graded algebras~\cite{C2}. Unless otherwise stated, all algebras are defined over a fixed commutative base ring of scalars $\mathcal{R}$.

\begin{definition}\label{def:BiHomAlg}
Let $\mathcal{M}$ be an $\mathcal{R}$-module endowed with a bilinear
product
\[
\star:\mathcal{M}\times\mathcal{M}\to\mathcal{M},\qquad
(x,y)\mapsto x\star y.
\]
We say that $(\mathcal{M},\star,\psi,\phi)$ is a
\emph{regular BiHom-associative algebra} if there exist two commuting
$\mathcal{R}$-linear automorphisms $\psi,\phi:\mathcal{M}\to\mathcal{M}$ such that
\[
\psi\circ\phi=\phi\circ\psi, \qquad
\psi(x\star y)=\psi(x)\star\psi(y),\qquad
\phi(x\star y)=\phi(x)\star\phi(y),
\]
and the BiHom-associativity identity
\[
\psi(x)\star(y\star z)=(x\star y)\star\phi(z)
\]
holds for all $x,y,z\in\mathcal{M}$. The maps $\psi$ and $\phi$ are
called the twisting maps. When no confusion arises, we shall write
$(\mathcal{M},\star,\psi,\phi)$ for this structure.
\end{definition}

\begin{remark}\label{rem:regular-bihom}
In the broader literature, the twisting maps of a BiHom-associative algebra
are usually required only to be linear maps. In the present paper, however,
the twisting maps are assumed to be automorphisms. Thus the
BiHom-associative algebras considered here are regular. For brevity, we may
occasionally use the shorter term ``BiHom-algebra'' for a regular
BiHom-associative algebra when no ambiguity can arise.
\end{remark}

We next record the notions of subalgebra and ideal in this setting.

\begin{definition}\label{def:subalgebra-ideal}
Let $(\mathcal{M},\star,\psi,\phi)$ be a regular BiHom-associative algebra
over $\mathcal{R}$. An $\mathcal{R}$-submodule
$\mathcal{N}\subseteq\mathcal{M}$ is a \emph{subalgebra} if
\[
\psi(\mathcal{N})\subseteq\mathcal{N},\qquad
\phi(\mathcal{N})\subseteq\mathcal{N},\qquad
\mathcal{N}\star\mathcal{N}\subseteq\mathcal{N}.
\]
A subalgebra $\mathfrak{I}\subseteq\mathcal{M}$ is a
\emph{two-sided ideal}, or simply an \emph{ideal}, if it is invariant under
the twisting maps, that is,
\[
\psi(\mathfrak{I})=\mathfrak{I},\qquad
\phi(\mathfrak{I})=\mathfrak{I},
\]
and
\[
\mathfrak{I}\star\mathcal{M}+\mathcal{M}\star\mathfrak{I}
\subseteq\mathfrak{I}.
\]
\end{definition}

We shall also use graded (not necessarily associative) algebras.

\begin{definition}\label{def:graded-algebra}
Let $(\mathcal{A},\cdot)$ be an $\mathcal{R}$-algebra which is not necessarily associative, and let $G$ be a commutative monoid (for instance, an abelian group) written additively.
We say that $\mathcal{A}$ is \textbf{$G$-graded} if there is a direct sum decomposition
\[
\mathcal{A}=\bigoplus_{g\in G}\mathcal{A}_g,\qquad \mathcal{A}_g\cdot\mathcal{A}_h\subseteq\mathcal{A}_{g+h}\ \ \text{for all }g,h\in G,
\]
where each $\mathcal{A}_g$ is an $\mathcal{R}$-submodule of $\mathcal{A}$.
\end{definition}

Let $\mathcal{M}_I(\mathcal{R})$ denote the algebra of all
$I\times I$ matrices over $\mathcal{R}$ with only finitely many
non-zero entries, where $I$ is an arbitrary index set. With the usual
matrix multiplication, denoted by juxtaposition, this is an associative
$\mathcal{R}$-algebra.

\begin{definition}\label{def:matrix-BiHom}
Let $\psi,\phi:\mathcal{M}_I(\mathcal{R})\to\mathcal{M}_I(\mathcal{R})$ be two commuting $\mathcal{R}$-algebra automorphisms of the usual matrix
algebra. Define a new product on $\mathcal{M}_I(\mathcal{R})$ by
\[
x\star y:=\psi(x)\,\phi(y),\qquad x,y\in\mathcal{M}_I(\mathcal{R}).
\]
Then $\big(\mathcal{M}_I(\mathcal{R}),\star,\psi,\phi\big)$ is a regular BiHom-associative algebra.
\end{definition}

Suppose now that $\mathcal{M}_I(\mathcal{R})$ is graded by an abelian group $(G,+)$:
\[
\mathcal{M}_I(\mathcal{R})=\bigoplus_{g\in G}\mathcal{M}_I(\mathcal{R})_g.
\]
Assume further that $\psi$ and $\phi$ are compatible with the grading in the sense that there exist commuting group automorphisms $\alpha,\beta:G\to G$ with
\begin{align}
\psi\big(\mathcal{M}_I(\mathcal{R})_g\big)&\subseteq \mathcal{M}_I(\mathcal{R})_{\alpha(g)},\label{eq:psi-compat}\\
\phi\big(\mathcal{M}_I(\mathcal{R})_g\big)&\subseteq \mathcal{M}_I(\mathcal{R})_{\beta(g)},\label{eq:phi-compat}
\end{align}
for all $g\in G$ (see~\cite{K}). In this situation we use the \emph{regular BiHom-group} notation
\[
\big(G,+_{(\alpha,\beta)},\alpha,\beta,\mathbf{0}\big),\qquad g_1+_{(\alpha,\beta)}g_2:=\alpha(g_1)+\beta(g_2),
\]
and view
$\big(\mathcal{M}_I(\mathcal{R}),\star,\psi,\phi\big)$ as graded by this
regular BiHom-group:
\[
\mathcal{M}_I(\mathcal{R})_{g_1}\star \mathcal{M}_I(\mathcal{R})_{g_2}\subseteq
\mathcal{M}_I(\mathcal{R})_{\,g_1+_{(\alpha,\beta)}g_2}\qquad(g_1,g_2\in G).
\]

\begin{definition}\label{def:graded-ideal}
A \textbf{graded ideal} of the graded regular BiHom-associative algebra $\big(\mathcal{M}_I(\mathcal{R}),\star,\psi,\phi\big)=\bigoplus_{g\in G}\mathcal{M}_I(\mathcal{R})_g$ is an $\mathcal{R}$-submodule $\mathfrak{I}\subseteq \mathcal{M}_I(\mathcal{R})$ such that
\[
\mathfrak{I}\star\mathcal{M}_I(\mathcal{R})+\mathcal{M}_I(\mathcal{R})\star\mathfrak{I}\subseteq \mathfrak{I},\qquad
\psi(\mathfrak{I})=\mathfrak{I},\qquad \phi(\mathfrak{I})=\mathfrak{I},
\]
and which admits a graded decomposition
\[
\mathfrak{I}=\bigoplus_{g\in G}\mathfrak{I}_g,\qquad \mathfrak{I}_g:=\mathfrak{I}\cap \mathcal{M}_I(\mathcal{R})_g.
\]
The graded regular BiHom-associative algebra is \textbf{graded simple} if $\mathcal{M}_I(\mathcal{R})\star\mathcal{M}_I(\mathcal{R})\neq 0$ and its only graded ideals are $\{\mathbf{0}\}$ and $\mathcal{M}_I(\mathcal{R})$.
\end{definition}

\begin{remark}
In much of the cited literature, the base ring is assumed to be a field. Our formulation over a general commutative ring $\mathcal{R}$ includes those cases. Later statements involving dimensions are to be understood over fields.
\end{remark}

\section{Equivalence relation: \texorpdfstring{$\sim$}{sim}}\label{sect3}

We consider the BiHom-algebra $(\mathcal{M}_{I}(\mathcal{R}), \star, \psi, \phi)$, which arises as a \emph{double twisting} of the classical matrix algebra $(\mathcal{M}_{I}(\mathcal{R}), \cdot)$ by means of two commuting algebra automorphisms $\psi$ and $\phi$. The new multiplication is defined as
\[
x \star y := \psi(x) \cdot \phi (y),
\]
for all $x, y \in \mathcal{M}_{I}(\mathcal{R})$, and equips the space with a BiHom-associative structure.

In addition, we endow $\mathcal{M}_{I}(\mathcal{R})$ with a grading over a \emph{regular BiHom-group} $(G, +_{(\alpha,\beta)}, \alpha, \beta)$, where $\alpha$ and $\beta$ are two commuting automorphisms of the underlying abelian group $(G, +)$, and the twisted group operation is given by
\[
g +_{(\alpha, \beta)} g' := \alpha(g) + \beta
(g').
\]
In this framework, we assume the existence of a decomposition
\begin{equation} \label{decomposition_formula}
\mathcal{M}_{I}(\mathcal{R}) = \bigoplus_{g \in G} \mathcal{M}_{I}(\mathcal{R})_
{g},
\end{equation}
where each $\mathcal{M}_{I}(\mathcal{R})_g$ is an $\mathcal{R}$-submodule associated with the degree $g$, and the grading is compatible with the BiHom-structure defined by the maps $\alpha$, $\beta$, $\psi$, and $\phi$.

And it is as the direct sum of $\mathcal{R}$-submodules satisfying 
\begin{align}\label{relation_between}
    \psi(\mathcal{M}_{I}(\mathcal{R})_{g}) \subset \mathcal{M}_{I}(\mathcal{R})_{\alpha(g)} \text{ and } \phi(\mathcal{M}_{I}(\mathcal{R})_{g}) \subset \mathcal{M}_{I}(\mathcal{R})_{\beta(g)}
\end{align}
with the relation 
\begin{align}\label{2.3}
\mathcal{M}_{I}(\mathcal{R})_{g} \star \mathcal{M}_{I}(\mathcal{R})_{g^{\prime}} \subset \mathcal{M}_{I}(\mathcal{R})_{\alpha(g) + \beta(g')}
\end{align}
for all $g, g' \in G$.

\begin{definition}
The \textbf{support} of the grading~\eqref{decomposition_formula} is defined as the set
\[ \Sigma := \{g \in G \setminus \{0\} ~|~ \mathcal{M}_{I}(\mathcal{R})_{g} \neq 0\}.\]
We say that the support is \textbf{symmetric} if, for every $g \in \Sigma$, it follows that $-g \in \Sigma$.
\end{definition}

\begin{remark}\label{power_properties}
From equation~\eqref{relation_between}, it follows that if $g \in \Sigma$, then every element of the form $\alpha^{i} \circ \beta^{j}(g)$, for all $i, j \in \mathbb{N}$, also belongs to $\Sigma$.
\end{remark}

In this section, we develop a notion of \emph{connections} within the support of the grading, which will serve as a central tool in the structural analysis of the graded BiHom-algebra $(\mathcal{M}_{I}(\mathcal{R}), \star, \psi, \phi)$.
Throughout the discussion, we assume that the grading of $(\mathcal{M}_{I}(\mathcal{R}), \star, \psi, \phi)$ admits a symmetric support $\Sigma$, and we shall refer to the decomposition
\[
\mathcal{M}_{I}(\mathcal{R}) = \bigoplus_{g \in G} \mathcal{M}_{I}(\mathcal{R})_{g} = \mathcal{M}_{I}(\mathcal{R})_{0} \oplus \left( \bigoplus_{g \in \Sigma} \mathcal{M}_{I}(\mathcal{R})_{g} \right)
\]
as the associated splitting of the algebra.

\begin{definition}\label{connection}
Let $g, g' \in \Sigma$. We say that $g$ is \textbf{connected} to $g'$ if there exists a finite sequence $\{g_1, g_2, \ldots, g_k\} \subseteq \Sigma$ such that the following conditions hold:

\begin{itemize}
    \item If $k = 1$, then
    \[
    g_1 \in \left\{\alpha^i \circ \beta^j(g) \mid i,j \in \mathbb{N} \right\} \cap \left\{ \pm \alpha^m \circ \beta^n(g') \mid m,n \in \mathbb{N} \right\}.
    \]

    \item If $k \geq 2$, then:
    \begin{enumerate}
        \item $g_1 \in \left\{ \alpha^i \circ \beta^j(g) \mid i,j \in \mathbb{N} \right\}$,\label{1}
        
        \item The successive iterated sums all belong to $\Sigma$.
        \[
        \begin{aligned}
        &\alpha(g_1) + \beta(g_2), \\
        &\alpha^2(g_1) + \alpha \circ \beta(g_2) + \beta(g_3), \\
        &\alpha^3(g_1) + \alpha^2 \circ \beta(g_2) + \alpha \circ \beta(g_3) + \beta(g_4), \\
        &\quad \vdots \\
        &\alpha^{k-2}(g_1) + \alpha^{k-3} \circ \beta(g_2) + \alpha^{k-4} \circ \beta(g_3) + \cdots + \beta(g_{k-1}).
        \end{aligned}
        \]\label{2}
        
        \item The total sum belongs to $\left\{ \pm \alpha^m \circ \beta^n(g') \mid m,n \in \mathbb{N} \right\}.\label{3}
        $
        \[
        \alpha^{k-1}(g_1) + \alpha^{k-2} \circ \beta(g_2) + \alpha^{k-3} \circ \beta(g_3) + \cdots + \beta(g_k).
        \]
    \end{enumerate}
\end{itemize}
In this case, we refer to the sequence $\{g_1, \ldots, g_k\}$ as a \textbf{connection} from $g$ to $g'$.
\end{definition}

\begin{proposition}\label{prop_case1}
In the case $k = 1$, $g$ is connected to $g'$ if and only if there exist $x, y \in \mathbb{Z}$ and $\epsilon \in \{ \pm 1 \}$ such that
\[g' = \epsilon \cdot \alpha^{x} \circ \beta^{y}(g).\]
\end{proposition}


\begin{lemma}\label{lemma_g}
For each $g \in \Sigma$, we have that $\alpha^{p}\circ\beta^{\overline{p}}(g)$ is connected to $\epsilon \alpha^{q}\circ\beta^{\overline{q}}(g)$ for all $p,\overline{p},q,\overline{q} \in \mathbb{N}$ and $\epsilon \in \{\pm 1\}$.
\end{lemma}

\begin{proof}
By Remark~\ref{power_properties}, both $\alpha^{p} \circ \beta^{\overline{p}}(g)$ and $\epsilon \cdot \alpha^{q} \circ \beta^{\overline{q}}(g)$ lie in $\Sigma$. Set $r = \max\{p, \overline{p}, q, \overline{q}\}$. Then, according to Proposition~\ref{prop_case1}, the singleton $\{ \alpha^{r} \circ \beta^{r}(g) \}$ forms a connection from $\alpha^{p} \circ \beta^{\overline{p}}(g)$ to $\epsilon \cdot \alpha^{q} \circ \beta^{\overline{q}}(g)$.
\end{proof}

\begin{lemma}\label{lemma_2.6}
Let $\{g_{1}, \ldots, g_{k}\}$ be a connection from $g$ to $g'$ such that $g_{1} = \alpha^{i} \circ \beta^{j}(g)$ for some $i, j \in \mathbb{N}$. Then, for all $r, s \in \mathbb{N}$ with $r \geq i$ and $s \geq j$, there exists a connection $\{\overline{g}_{1}, \ldots, \overline{g}_{k}\}$ from $g$ to $g'$ such that $\overline{g}_{1} = \alpha^{r} \circ \beta^{s}(g)$.
\end{lemma}

\begin{proof}
By Remark~\ref{power_properties}, we have
\[
\left\{ \alpha^{r-i} \circ \beta^{s-j}(g_1), \ldots, \alpha^{r-i} \circ \beta^{s-j}(g_k) \right\} \subseteq \Sigma.
\]
Define $\overline{g}_\ell := \alpha^{r-i} \circ \beta^{s-j}(g_\ell)$ for each $\ell = 1, \ldots, k$. Then, using Remark~\ref{power_properties} and Definition~\ref{connection}, it is straightforward to verify that $\{\overline{g}_{1}, \ldots, \overline{g}_{k}\}$ is a connection from $g$ to $g'$, with
$
\overline{g}_{1} = \alpha^{r-i} \circ \beta^{s-j} \left( \alpha^{i} \circ \beta^{j}(g) \right) = \alpha^{r} \circ \beta^{s}(g).
$
\end{proof}

\begin{lemma}\label{lemma_2.7}
Let $\{g_{1}, \ldots, g_{k}\}$ be a connection from $g$ to $g'$, satisfying one of the following:

\begin{itemize}
    \item If $k = 1$, then $g_1 = \epsilon \cdot \alpha^{i} \circ \beta^{j}(g')$;
    
    \item If $k \geq 2$, then
    \[
    \alpha^{k-1}(g_1) + \alpha^{k-2} \circ \beta(g_2) + \alpha^{k-3} \circ \beta(g_3) + \cdots + \beta(g_k) = \epsilon \cdot \alpha^{i} \circ \beta^{j}(g'),
    \]
\end{itemize}
where $i, j \in \mathbb{N}$ and $\epsilon \in \{ \pm 1 \}$.

Then, for any $r, s \in \mathbb{N}$ such that $r \geq i$ and $s \geq j$, there exists a connection $\{\overline{g}_{1}, \ldots, \overline{g}_{k}\}$ from $g$ to $g'$ such that:

\begin{itemize}
    \item If $k = 1$, then $\overline{g}_1 = \epsilon \cdot \alpha^{r} \circ \beta^{s}(g')$;
    
    \item If $k \geq 2$, then
    \[
    \alpha^{k-1}(\overline{g}_1) + \alpha^{k-2} \circ \beta(\overline{g}_2) + \alpha^{k-3} \circ \beta(\overline{g}_3) + \cdots + \beta(\overline{g}_k) = \epsilon \cdot \alpha^{r} \circ \beta^{s}(g').
    \]
\end{itemize}
\end{lemma}

\begin{proof}
By Remark~\ref{power_properties}, we have
\[
\left\{ \alpha^{r-i} \circ \beta^{s-j}(g_1), \ldots, \alpha^{r-i} \circ \beta^{s-j}(g_k) \right\} \subseteq \Sigma.
\]
Define $\overline{g}_\ell := \alpha^{r-i} \circ \beta^{s-j}(g_\ell)$ for $\ell = 1, \ldots, k$. It follows again from Remark~\ref{power_properties} and Definition~\ref{connection} that the sequence $\{ \overline{g}_1, \ldots, \overline{g}_k \}$ forms a connection from $g$ to $g'$.

In the case $k = 1$, we immediately obtain
\[
\overline{g}_1 = \alpha^{r-i} \circ \beta^{s-j} \left( \epsilon \cdot \alpha^i \circ \beta^j(g') \right) = \epsilon \cdot \alpha^r \circ \beta^s(g').
\]

If $k \geq 2$, then we compute:
\begin{align*}
&\alpha^{k-1}(\overline{g}_1) + \alpha^{k-2} \circ \beta(\overline{g}_2) + \alpha^{k-3} \circ \beta(\overline{g}_3) + \cdots + \beta(\overline{g}_k) \\
&= \alpha^{r-i} \circ \beta^{s-j} \left( \alpha^{k-1}(g_1) + \alpha^{k-2} \circ \beta(g_2) + \cdots + \beta(g_k) \right) \\
&= \alpha^{r-i} \circ \beta^{s-j} \left( \epsilon \cdot \alpha^i \circ \beta^j(g') \right) = \epsilon \cdot \alpha^r \circ \beta^s(g').
\end{align*}
\end{proof}

\begin{theorem}\label{equivalence_relation}
The relation $\sim$ on $\Sigma$, defined by $g \sim g'$ if and only if $g$ is connected to $g'$, is an equivalence relation.
\end{theorem}

\begin{proof}
     For each $g \in \Sigma$, Lemma~\ref{lemma_g} can give us that $g \sim g$ and so $\sim$ is reflexive.\\
Let us see the symmetric character of $\sim$. If $g \sim g^{\prime}$, there exists a connection
\begin{equation}\label{2.4}
\left\{g_{1}, g_{2}, \ldots, g_{k}\right\} \subset \Sigma 
\end{equation}
from $g$ to $g^{\prime}$.

If $k=1$, we have $g_{1}=\alpha^{i}\circ\beta^{i'}(g)$ and $g_{1}=\epsilon \alpha^{j}\circ\beta^{j'}\left(g^{\prime}\right)$ with $i,i',j,j' \in \mathbb{N}$ and $\epsilon \in \{\pm 1\}$. From here, it is clear that $\left\{\epsilon g_{1}\right\}$ is a connection from $g^{\prime}$ to $g$ and so $g^{\prime} \sim g$.

If $k \geq 2$, we have that the connection~\eqref{2.4} satisfies conditions ~\eqref{1},~\eqref{2} and ~\eqref{3} in Definition~\ref{connection}. First we observe the condition~\eqref{3}, let us distinguish two possibilities. In the first one
\begin{align}\label{2.5}
&\quad\alpha^{k-1}\left(g_{1}\right)+\alpha^{k-2}\circ\beta\left(g_{2}\right)+\alpha^{k-3}\circ\beta\left(g_{3}\right)+\cdots\\ \nonumber
&+\alpha^{i-1}\circ\beta\left(g_{k-i+1}\right)
+\cdots+\beta\left(g_{k}\right)=\alpha^{i}\circ\beta^j\left(g^{\prime}\right),
\end{align}
and in the second one
\begin{align}\label{2.6}
&\quad\alpha^{k-1}\left(g_{1}\right)+\alpha^{k-2}\circ\beta\left(g_{2}\right)+\alpha^{k-3}\circ\beta\left(g_{3}\right)+\cdots\\ \nonumber
&+\alpha^{i-1}\circ\beta\left(g_{k-i+1}\right)
+\cdots+\beta\left(g_{k}\right)=-\alpha^{i}\circ\beta^j\left(g^{\prime}\right),
\end{align}
for some $i,j \in \mathbb{N}$. By Lemma~\ref{lemma_2.7}, replacing the given connection by a shifted one
if necessary, we may assume in what follows that $i\geq 1$.

Suppose we have the first possibility~\eqref{2.5}. By Remark~\ref{power_properties} and the symmetry of $\Sigma$, we can consider the set
\begin{align}\label{2.7}
\{\alpha^{i-1}\circ\beta^j\left(g^{\prime}\right),-g_{k},-\alpha^{2}\left(g_{k-1}\right), \ldots,
-\alpha^{2i}\left(g_{k-i}\right), \ldots,-\alpha^{2 k-4}\left(g_{2}\right)\} \subset \Sigma
\end{align}

Let us show that this set is a connection from $g^{\prime}$ to $g$. It is clear that the set ~\eqref{2.7} satisfies condition~\eqref{1} of Definition~\ref{connection} so let us verify that the set~\eqref{2.7} satisfies condition~\eqref{2}. We have
\begin{align*}
&\quad\alpha\left(\alpha^{i-1}\circ\beta^j\left(g^{\prime}\right)\right)-\beta\left(g_{k}\right)
=\alpha^{i}\circ\beta^j\left(g^{\prime}\right)-\beta\left(g_{k}\right)\\
&= 
\alpha^{k-1}\left(g_{1}\right)+\alpha^{k-2}\circ\beta\left(g_{2}\right)+\alpha^{k-3}\circ\beta\left(g_{3}\right)+\cdots+\alpha^{k-i}\circ\beta\left(g_{i}\right)+\cdots+\alpha\circ\beta\left(g_{k-1}\right)\\
&=\alpha\left(\alpha^{k-2}\left(g_{1}\right)+\alpha^{k-3}\circ\beta\left(g_{2}\right)+\alpha^{k-4}\circ\beta\left(g_{3}\right)+\cdots+\alpha^{k-i-1}\circ\beta\left(g_{i}\right)+\cdots+\beta\left(g_{k-1}\right)\right),
\end{align*}
where the third equality is a consequence of equation~\eqref{2.5}.

By condition~\eqref{2} of Definition~\ref{connection} applied to the connection~\eqref{2.4}, we can see that $\alpha^{k-2}\left(g_{1}\right)+\alpha^{k-3}\circ\beta\left(g_{2}\right)+\alpha^{k-4}\circ\beta\left(g_{3}\right)+\cdots+\alpha^{k-i-1}\circ\beta\left(g_{i}\right)+\cdots+\beta\left(g_{k-1}\right) \in \Sigma$. Then using Remark~\ref{power_properties}, we deduce that 
$$
\alpha\left(\alpha^{i-1}\circ\beta^j\left(g^{\prime}\right)\right)-\beta\left(g_{k}\right)\in\Sigma
$$

For each $1 \leq i \leq k-2$ we also have that
\begin{align*}
&\quad\alpha^{i}\left(\alpha^{i-1}\circ\beta^j\left(g^{\prime}\right)\right)-\alpha^{i-1}\circ\beta\left(g_{k}\right)-\alpha^{i-2}\circ\beta\left(\alpha^{2}\left(g_{k-1}\right)\right)-\cdots-\beta\left(\alpha^{2 i-2}\left(g_{k-(i-1)}\right)\right)\\
&=\alpha^{i-1}\left(\alpha^{i}\circ\beta^j\left(g^{\prime}\right)-\beta\left(g_{k}\right)-\alpha\circ\beta\left(g_{k-1}\right)-\cdots-\alpha^{i-1}\circ\beta\left(g_{k-(i-1)}\right)\right)\\
&=\alpha^{i-1}\left(\alpha^{k-1}\left(g_{1}\right)+\alpha^{k-2}\circ\beta\left(g_{2}\right)+\cdots+\alpha^{i}\circ\beta\left(g_{k-i}\right)\right)\\
&=\alpha^{2 i-1}\left(\alpha^{k-1-i}\left(g_{1}\right)+\alpha^{k-2-i}\circ\beta\left(g_{2}\right)+\cdots+\beta\left(g_{k-i}\right)\right),
\end{align*}
with the third equality being a consequence of equation~\eqref{2.5}. Taking into account that, by condition~\eqref{2} of Definition~\ref{connection} applied to~\eqref{2.4},
$$
\alpha^{k-1-i}\left(g_{1}\right)+\alpha^{k-2-i}\circ\beta\left(g_{2}\right)+\cdots+\beta\left(g_{k-i}\right) \in \Sigma.
$$
Then we get as consequence of Remark~\ref{power_properties} that
$$
\alpha^{i}\left(\alpha^{i-1}\circ\beta^j\left(g^{\prime}\right)\right)-\alpha^{i-1}\circ\beta\left(g_{k}\right)-\alpha^{i-2}\circ\beta\left(\alpha^{2}\left(g_{k-1}\right)\right)-\cdots-\beta\left(\alpha^{2 i-2}\left(g_{k-(i-1)}\right)\right) \in \Sigma.
$$
We have showed that the set~\eqref{2.7} satisfies condition~\eqref{2} of Definition~\ref{connection}. It just remains to prove that this set also satisfies condition~\ref{3} of this definition. We have as above that

\begin{align*}
&\quad\alpha^{k-1}\left(\alpha^{i-1}\circ\beta^j\left(g^{\prime}\right)\right)-\alpha^{k-2}\circ\beta\left(g_{k}\right)-\alpha^{k-3}\circ\beta\left(\alpha^{2}\left(g_{k-1}\right)\right)-\cdots-\beta\left(\alpha^{2 k-4}\left(g_{2}\right)\right)\\
&=\alpha^{k-2}\left(\alpha^{i}\circ\beta^j\left(g^{\prime}\right)-\beta\left(g_{k}\right)-\alpha\circ\beta\left(g_{k-1}\right)-\cdots-\alpha^{k-2}\circ\beta\left(g_{2}\right)\right)\\
&=\alpha^{k-2}\left(\alpha^{k-1}\left(g_{1}\right)\right)\\
&=\alpha^{2k-3}(g_1).
\end{align*}

Condition~\eqref{1} of Definition~\ref{connection} applied to the connection~\eqref{2.4} gives us that $g_{1}=\alpha^{i}\circ\beta^j(g)$ for some $i,j \in \mathbb{N}$ and so
\begin{align*}
&\quad\alpha^{k-1}\left(\alpha^{i-1}\circ\beta^j\left(g^{\prime}\right)\right)-\alpha^{k-2}\circ\beta\left(g_{k}\right)-\alpha^{k-3}\circ\beta\left(\alpha^{2}\left(g_{k-1}\right)\right)-\cdots-\beta\left(\alpha^{2 k-4}\left(g_{2}\right)\right)\\
&=\alpha^{2k-3}(g_1) \in\left\{\alpha^{m}\beta^n(g)~|~ m,n \in \mathbb{N}\right\}
\end{align*}

We have showed that the set~\eqref{2.7} is actually a connection from $g^{\prime}$ to $g$.\\
Now suppose we are in the second possibility given by equation~\eqref{2.6}. Then we can prove as in the above first possibility, given by equation~\eqref{2.5}, that
$$
\{\alpha^{i-1}\circ\beta^j\left(g^{\prime}\right),g_{k},\alpha^{2}\left(g_{k-1}\right), \ldots,
\alpha^{2i}\left(g_{k-i}\right), \ldots,\alpha^{2 k-4}\left(g_{2}\right)\} \subset \Sigma
$$
is a connection from $g^{\prime}$ to $g$. We conclude $g^{\prime} \sim g$ and so the relation $\sim$ is symmetric.

Finally, let us verify that $\sim$ is transitive. Suppose $g \sim g^{\prime}$ and $g^{\prime} \sim g^{\prime \prime}$, and write $\left\{g_{1}, \ldots, g_{k}\right\}$ for a connection from $g$ to $g^{\prime}$ and $\left\{\overline{g}_{1}, \ldots, \overline{g}_{r}\right\}$ for a connection from $g^{\prime}$ to $g^{\prime \prime}$. From here, we have

\begin{equation}\label{2.8}
\left\{
\begin{aligned}
& g_{1}=\epsilon \alpha^{m}\circ\beta^n\left(g^{\prime}\right), \text { when } k=1 \\
& \alpha^{k-1}(g_{1}) + \alpha^{k-2}\circ\beta(g_{2}) + \alpha^{k-3}\circ\beta(g_{3}) + \cdots+\\
&\alpha^{k-i}\circ\beta(g_{i}) + \cdots + \beta(g_{k})=\epsilon\alpha^{m}\circ\beta^n(g^{\prime}), \text { when } k\geq 2
\end{aligned}
\right.
\end{equation}

for some $m,n \in \mathbb{N}, \epsilon \in \{\pm 1\}$. And
\begin{align}\label{2.9}
\overline{g}_{1}=\alpha^{q}\circ\beta^s\left(g^{\prime}\right)
\end{align}
for some $q,s \in \mathbb{N}$. By Lemmas~\ref{lemma_2.6} and~\ref{lemma_2.7}, after replacing the two connections by shifted connections if necessary, we may assume that $m=q$ and $n=s$.

In the case $r=1$, suppose $m=q$ and $n=s$, we have $\overline{g}_{1}=\tau \alpha^{t}\circ\beta^u\left(g^{\prime \prime}\right)$ with $t,u \in \mathbb{N}$ and $\tau \in \{\pm 1\}$. Since $m=q$ and $n=s$, then $g_{1}=\epsilon\alpha^{m}\circ\beta^n\left(g^{\prime}\right)=\epsilon \overline{g}_{1}=\epsilon \tau \alpha^{t}\circ\beta^u\left(g^{\prime \prime}\right)$ if $k=1$, and $
\alpha^{k-1}(g_{1}) + \alpha^{k-2}\circ\beta(g_{2}) + \alpha^{k-3}\circ\beta(g_{3}) + \cdots+\alpha^{k-i}\circ\beta(g_{i}) + \cdots + \beta(g_{k})=\epsilon\alpha^{m}\circ\beta^n(g^{\prime})=\epsilon \overline{g}_{1}=\epsilon \tau \alpha^{t}\circ\beta^u\left(g^{\prime \prime}\right),
$ if $k \geq 2$. From here, we get that $\left\{g_{1}, \ldots, g_{k}\right\}$ is also a connection from $g$ to $g^{\prime \prime}$.

In the case $r \geq 2$, it is straightforward to verify, taking into account equations~\eqref{2.8} and~\eqref{2.9}, and the fact $m=q$ and $n=s$, that $\left\{g_{1}, \ldots, g_{k}, \overline{g}_{2}, \ldots, \overline{g}_{r}\right\}$ is a connection from $g$ to $g^{\prime \prime}$ if $\epsilon=+$ in equation~\eqref{2.8}. Taking also into account the symmetry of $\Sigma$, that $\left\{g_{1}, \ldots, g_{k},-\overline{g}_{2}, \ldots,-\overline{g}_{r}\right\}$ it is if $\epsilon=-$ in equation~\eqref{2.8}. We have showed the connection relation is also transitive and so it is an equivalence relation.
\end{proof}

\section{Decompositions of Matrix algebras}\label{sect4}

Consider the quotient set
\[
\Sigma / \sim := \{ [g] \mid g \in \Sigma \}.
\]
In order to associate a suitable graded ideal $\mathfrak{I}_{[g]}$ of the graded BiHom-algebra $\left(\mathcal{M}_{I}(\mathcal{R}), \star, \psi, \phi\right)$ to each equivalence class $[g] \in \Sigma / \sim$ in this section, we define
\[
\mathfrak{I}_{\mathbf{0},[g]} := \operatorname{Span}_{\mathcal{R}} \left\{ \mathcal{M}_{I}(\mathcal{R})_{\beta(g')} \star \mathcal{M}_{I}(\mathcal{R})_{-\alpha(g')} \mid g' \in [g] \right\} \subseteq \mathcal{M}_{I}(\mathcal{R})_{\mathbf{0}},
\]
and
\[
V_{[g]} := \bigoplus_{g' \in [g]} \mathcal{M}_{I}(\mathcal{R})_{g'}.
\]
Finally, we define
\[
\mathfrak{I}_{[g]} := \mathfrak{I}_{\mathbf{0},[g]} \oplus V_{[g]}.
\]

Since each equivalence class $[g]$ is invariant under
$\alpha^m\circ\beta^n$ for all $m,n\in\mathbb{Z}$, we shall also use the
equivalent description
\[
\mathfrak{I}_{\mathbf{0},[g]}
=
\operatorname{Span}_{\mathcal{R}}
\left\{
\mathcal{M}_{I}(\mathcal{R})_{h}
\star
\mathcal{M}_{I}(\mathcal{R})_{-\beta^{-1}\circ\alpha(h)}
\mid h\in[g]
\right\}.
\]

\begin{lemma}\label{lemma_subideal}
Let $g \in \Sigma$. We have $\mathfrak{I}_{[g]} \star \mathfrak{I}_{[g]} \subseteq \mathfrak{I}_{[g]}$,
$\psi\left(\mathfrak{I}_{[g]}\right)=\mathfrak{I}_{[g]}$, and
$\phi\left(\mathfrak{I}_{[g]}\right)=\mathfrak{I}_{[g]}$.
\end{lemma}

\begin{proof}

We begin by noting the identity
\begin{align}\label{eq:self-product-expansion}
\left( \mathfrak{I}_{\mathbf{0},[g]} \oplus V_{[g]} \right)
\star
\left( \mathfrak{I}_{\mathbf{0},[g]} \oplus V_{[g]} \right)
&\subseteq
\mathfrak{I}_{\mathbf{0},[g]}\star\mathfrak{I}_{\mathbf{0},[g]}
+
\mathfrak{I}_{\mathbf{0},[g]}\star V_{[g]} \nonumber\\
&\quad
+
V_{[g]}\star\mathfrak{I}_{\mathbf{0},[g]}
+
V_{[g]}\star V_{[g]}.
\end{align}

We first consider the final summand in equation~\eqref{eq:self-product-expansion}, namely $V_{[g]} \star V_{[g]}$. Let $g', g'' \in [g]$ be such that
\[
\mathcal{M}_{I}(\mathcal{R})_{g'} \star \mathcal{M}_{I}(\mathcal{R})_{g''} \neq \mathbf{0}.
\]
If $g'' = -\beta^{-1} \circ \alpha(g')$, then clearly
\[
\mathcal{M}_{I}(\mathcal{R})_{g'} \star \mathcal{M}_{I}(\mathcal{R})_{g''} 
= \mathcal{M}_{I}(\mathcal{R})_{g'} \star \mathcal{M}_{I}(\mathcal{R})_{-\beta^{-1} \circ \alpha(g')} 
\subseteq \mathfrak{I}_{\mathbf{0},[g]}.
\]

On the other hand, if
$g'' \neq -\beta^{-1}\circ\alpha(g')$, then $\alpha(g')+\beta(g'')\neq \mathbf{0}$.
Since
\[
\mathcal{M}_{I}(\mathcal{R})_{g'} \star
\mathcal{M}_{I}(\mathcal{R})_{g''} \neq \mathbf{0},
\]
equation~\eqref{2.3} gives
\[
\alpha(g')+\beta(g'')\in\Sigma.
\]
It follows that $\{g', g''\}$ forms a connection from $g'$ to $\alpha(g') + \beta(g'')$, and hence, by the transitivity of $\sim$, we conclude that $\alpha(g') + \beta(g'') \in [g]$. Thus,
\[
\mathcal{M}_{I}(\mathcal{R})_{g'} \star \mathcal{M}_{I}(\mathcal{R})_{g''} 
\subseteq \mathcal{M}_{I}(\mathcal{R})_{\alpha(g') + \beta(g'')} 
\subseteq V_{[g]}.
\]

Consequently, we obtain
\[
\left( \bigoplus_{g' \in [g]} \mathcal{M}_{I}(\mathcal{R})_{g'} \right) 
\star 
\left( \bigoplus_{g'' \in [g]} \mathcal{M}_{I}(\mathcal{R})_{g''} \right) 
\subseteq \mathfrak{I}_{\mathbf{0},[g]} \oplus V_{[g]},
\]
that is,
\begin{equation}\label{3.2}
V_{[g]} \star V_{[g]} \subseteq \mathfrak{I}_{\mathbf{0},[g]} \oplus V_{[g]}.
\end{equation}

Next, we consider the first summand $\mathfrak{I}_{\mathbf{0},[g]}\star \mathfrak{I}_{\mathbf{0},[g]}$ in equation~\eqref{eq:self-product-expansion}. By BiHom-associativity together with the grading compatibility
conditions~\eqref{eq:psi-compat} and~\eqref{eq:phi-compat}, for any $g',g''\in[g]$, we have
\begin{align*}
&\quad \left( \mathcal{M}_{I}(\mathcal{R})_{g'} \star \mathcal{M}_{I}(\mathcal{R})_{-\beta^{-1} \circ \alpha(g')} \right) 
\star 
\left( \mathcal{M}_{I}(\mathcal{R})_{g''} \star \mathcal{M}_{I}(\mathcal{R})_{-\beta^{-1} \circ \alpha(g'')} \right) \\
&= \left( \mathcal{M}_{I}(\mathcal{R})_{g'} \star \mathcal{M}_{I}(\mathcal{R})_{-\beta^{-1} \circ \alpha(g')} \right)
\star 
\phi\left( \phi^{-1} \left( \mathcal{M}_{I}(\mathcal{R})_{g''} \star \mathcal{M}_{I}(\mathcal{R})_{-\beta^{-1} \circ \alpha(g'')} \right) \right) \\
&\subseteq \psi\left( \mathcal{M}_{I}(\mathcal{R})_{g'} \right)
\star 
\left( \mathcal{M}_{I}(\mathcal{R})_{-\beta^{-1} \circ \alpha(g')} 
\star 
\phi^{-1} \left( \mathcal{M}_{I}(\mathcal{R})_{g''} \star \mathcal{M}_{I}(\mathcal{R})_{-\beta^{-1} \circ \alpha(g'')} \right) \right) \\
&\subseteq \mathcal{M}_{I}(\mathcal{R})_{\alpha(g')} 
\star 
\mathcal{M}_{I}(\mathcal{R})_{-\beta^{-1} \circ \alpha^2(g')} 
\subseteq \mathfrak{I}_{\mathbf{0},[g]}.
\end{align*}

Hence,
\begin{equation}\label{3.3}
\mathfrak{I}_{\mathbf{0},[g]} \star \mathfrak{I}_{\mathbf{0},[g]} \subseteq \mathfrak{I}_{\mathbf{0},[g]}.
\end{equation}

Similarly, one can show that
\begin{equation}\label{3.4}
\mathfrak{I}_{\mathbf{0},[g]} \star V_{[g]} + V_{[g]} \star \mathfrak{I}_{\mathbf{0},[g]} \subseteq V_{[g]}.
\end{equation}

Finally, combining equations~\eqref{eq:self-product-expansion}, \eqref{3.2}, \eqref{3.3} and~\eqref{3.4}, we obtain
\[
\mathfrak{I}_{[g]} \star \mathfrak{I}_{[g]} 
= \left( \mathfrak{I}_{\mathbf{0},[g]} \oplus V_{[g]} \right) 
\star 
\left( \mathfrak{I}_{\mathbf{0},[g]} \oplus V_{[g]} \right) 
\subseteq \mathfrak{I}_{[g]}.
\]

Now, from the grading compatibility conditions~\eqref{eq:psi-compat}
and~\eqref{eq:phi-compat}, together with the fact that each equivalence
class $[g]$ is invariant under $\alpha^m\circ\beta^n$ for all
$m,n\in\mathbb{Z}$, it follows that
\[
\psi\left( \mathfrak{I}_{\mathbf{0},[g]} \right)
=
\mathfrak{I}_{\mathbf{0},[g]},
\quad
\psi\left(V_{[g]}\right)
=
V_{[g]},
\quad
\phi\left( \mathfrak{I}_{\mathbf{0},[g]} \right)
=
\mathfrak{I}_{\mathbf{0},[g]},
\quad
\phi\left(V_{[g]}\right)
=
V_{[g]}.
\]
Hence, we conclude that
\[
\psi\left( \mathfrak{I}_{[g]} \right)=\mathfrak{I}_{[g]},
\qquad
\phi\left( \mathfrak{I}_{[g]} \right)=\mathfrak{I}_{[g]}.
\qedhere
\]
\end{proof}

\begin{lemma}\label{lem:orthogonal-classes}
    Let $g \in \Sigma$. If $g^{\prime} \notin[g]$ then $\mathfrak{I}_{[g]} \star \mathfrak{I}_{\left[g^{\prime}\right]}=\mathbf{0}$.
\end{lemma}

\begin{proof}

We now consider the product between ideals corresponding to distinct equivalence classes. From equation~\eqref{eq:orthogonal-expansion}, We have
\begin{align}\label{eq:orthogonal-expansion}
\left( \mathfrak{I}_{\mathbf{0},[g]} \oplus V_{[g]} \right) 
\star 
\left( \mathfrak{I}_{\mathbf{0},[g']} \oplus V_{[g']} \right) 
&\subseteq 
\mathfrak{I}_{\mathbf{0},[g]} \star \mathfrak{I}_{\mathbf{0},[g']} 
+ \mathfrak{I}_{\mathbf{0},[g]} \star V_{[g']} \nonumber\\
&\quad
+ V_{[g]} \star \mathfrak{I}_{\mathbf{0},[g']} 
+ V_{[g]} \star V_{[g']}.
\end{align}

We first analyse the fourth summand, $V_{[g]} \star V_{[g']}$. Suppose, for contradiction, that there exist $g_1 \in [g]$ and $g_2 \in [g']$ such that
\[
\mathcal{M}_{I}(\mathcal{R})_{g_1} \star \mathcal{M}_{I}(\mathcal{R})_{g_2} \neq \mathbf{0}.
\]
If $g_2=-\beta^{-1}\circ\alpha(g_1)$, then $g_2=-\alpha\circ\beta^{-1}(g_1)$ because $\alpha$ and $\beta$ commute.
Hence, Proposition~\ref{prop_case1} gives $g_1\sim g_2$, contradicting $g_1\in[g]$ and $g_2\in[g']$ with $[g]\neq[g']$.
Therefore
\[
g_2\neq -\beta^{-1}\circ\alpha(g_1).
\]
Since $\mathcal{M}_{I}(\mathcal{R})_{g_1}\star
\mathcal{M}_{I}(\mathcal{R})_{g_2}\neq \mathbf{0}$, equation~\eqref{2.3} gives
\[
\alpha(g_1)+\beta(g_2)\in\Sigma.
\]
Therefore, the set $\{ g_1, g_2, -\beta^{-1} \circ \alpha^2(g_1) \}$ forms a connection between $g_1$ and $g_2$. By the transitivity of the connection relation, we deduce that $g \sim g'$, i.e., $[g] = [g']$, which contradicts the assumption that they belong to distinct equivalence classes.

Hence, for all $g_1 \in [g]$, $g_2 \in [g']$, we have
\[
\mathcal{M}_{I}(\mathcal{R})_{g_1} \star \mathcal{M}_{I}(\mathcal{R})_{g_2} = \mathbf{0},
\]
and so
\begin{equation}\label{3.6}
V_{[g]} \star V_{[g']} = \mathbf{0}.
\end{equation}

Next, we examine the first summand $\mathfrak{I}_{\mathbf{0},[g]} \star \mathfrak{I}_{\mathbf{0},[g']}$ in equation~\eqref{eq:orthogonal-expansion}. Suppose, for contradiction, that there exist $g_1 \in [g]$ and $g_2 \in [g']$ such that
\[
\left( \mathcal{M}_{I}(\mathcal{R})_{g_1} \star \mathcal{M}_{I}(\mathcal{R})_{-\beta^{-1} \circ \alpha(g_1)} \right) \star \left( \mathcal{M}_{I}(\mathcal{R})_{g_2} \star \mathcal{M}_{I}(\mathcal{R})_{-\beta^{-1} \circ \alpha(g_2)} \right) \neq 0.
\]

By BiHom-associativity and the same argument as in item (1), we obtain:
\[
\psi\left( \mathcal{M}_{I}(\mathcal{R})_{g_1} \right) \star \left( \mathcal{M}_{I}(\mathcal{R})_{-\beta^{-1} \circ \alpha(g_1)} \star \phi^{-1} \left( \mathcal{M}_{I}(\mathcal{R})_{g_2} \star \mathcal{M}_{I}(\mathcal{R})_{-\beta^{-1} \circ \alpha(g_2)} \right) \right) \neq 0.
\]

Applying BiHom-associativity again, we compute:
\begin{align*}
&\qquad \mathcal{M}_{I}(\mathcal{R})_{-\beta^{-1} \circ \alpha(g_1)} \star \phi^{-1} \left( \mathcal{M}_{I}(\mathcal{R})_{g_2} \star \mathcal{M}_{I}(\mathcal{R})_{-\beta^{-1} \circ \alpha(g_2)} \right) \neq 0 \\
&\Leftrightarrow \mathcal{M}_{I}(\mathcal{R})_{-\beta^{-1} \circ \alpha(g_1)} \star \left( \phi^{-1}\left( \mathcal{M}_{I}(\mathcal{R})_{g_2} \right) \star \phi^{-1}\left( \mathcal{M}_{I}(\mathcal{R})_{-\beta^{-1} \circ \alpha(g_2)} \right) \right) \neq 0 \\
&\Leftrightarrow \left( \psi^{-1}\left( \mathcal{M}_{I}(\mathcal{R})_{-\beta^{-1} \circ \alpha(g_1)} \right) \star \phi^{-1}\left( \mathcal{M}_{I}(\mathcal{R})_{g_2} \right) \right) \star \mathcal{M}_{I}(\mathcal{R})_{-\beta^{-1} \circ \alpha(g_2)} \neq 0.
\end{align*}

It follows that
\[
\psi^{-1}\left( \mathcal{M}_{I}(\mathcal{R})_{-\beta^{-1} \circ \alpha(g_1)} \right) \star \phi^{-1}\left( \mathcal{M}_{I}(\mathcal{R})_{g_2} \right) \subseteq \mathcal{M}_{I}(\mathcal{R})_{-\beta^{-1}(g_1)} \star \mathcal{M}_{I}(\mathcal{R})_{\beta^{-1}(g_2)} \neq 0,
\]
so that
\[
\mathcal{M}_{I}(\mathcal{R})_{-g_1} \star \mathcal{M}_{I}(\mathcal{R})_{g_2} \neq 0,
\]
which contradicts equation~\eqref{3.6}.

We therefore conclude that
\begin{equation}\label{3.7}
\mathfrak{I}_{\mathbf{0},[g]} \star \mathfrak{I}_{\mathbf{0},[g']} = 0.
\end{equation}

By similar reasoning, one also has:
\begin{equation}\label{3.8}
\mathfrak{I}_{\mathbf{0},[g]} \star V_{[g']} + V_{[g]} \star \mathfrak{I}_{\mathbf{0},[g']} = 0.
\end{equation}

Finally, by equation~\eqref{eq:orthogonal-expansion} we get
$$
\left(\mathfrak{I}_{\mathbf{0},[g]} \oplus V_{[g]}\right) \star\left(\mathfrak{I}_{\mathbf{0},\left[g^{\prime}\right]} \oplus V_{\left[g^{\prime}\right]}\right)=\mathbf{0},
$$
and so $\mathfrak{I}_{[g]} \star \mathfrak{I}_{\left[g^{\prime}\right]}=\mathbf{0}$. 
\end{proof}

From Lemma~\ref{lemma_subideal}, together with Definitions~\ref{def:subalgebra-ideal} and~\ref{def:graded-algebra}, it follows immediately that, for each $g\in\Sigma$, the subspace
$\mathfrak{I}_{[g]}$ is a graded subalgebra of $\left(\mathcal{M}_{I}(\mathcal{R}),\star,\psi,\phi\right)$.
We now proceed to show that, in fact, each $\mathfrak{I}_{[g]}$ is a
graded ideal of $\left(\mathcal{M}_{I}(\mathcal{R}),\star,\psi,\phi\right)$.

\begin{theorem}\label{lemma_idealofgraded}
For each $g \in \Sigma$, the graded $\mathcal{R}$-submodule
\[
\mathfrak{I}_{[g]} = \mathfrak{I}_{\mathbf{0},[g]} \oplus V_{[g]},
\]
associated with the equivalence class $[g]$, is a graded ideal of the BiHom-algebra $\left( \mathcal{M}_{I}(\mathcal{R}), \star, \psi, \phi \right)$.
\end{theorem}

\begin{proof}
 For each $g^{\prime} \in[g]$ we have $\mathcal{M}_{I}(\mathcal{R})_{g^{\prime}} \star \mathcal{M}_{I}(\mathcal{R})_{\mathbf{0}} \subset \mathcal{M}_{I}(\mathcal{R})_{\alpha\left(g^{\prime}\right)} \subset V_{[g]}$ and so
\begin{align}\label{eq:ideal-right-zero-from-V}
V_{[g]} \star \mathcal{M}_{I}(\mathcal{R})_{\mathbf{0}} \subseteq V_{[g]}.
\end{align}
By BiHom-associativity, we also have
\begin{align*}
&\quad\left(\mathcal{M}_{I}(\mathcal{R})_{g^{\prime}} \star \mathcal{M}_{I}(\mathcal{R})_{-\beta^{-1}\circ\alpha(g^{\prime})}\right) \star \mathcal{M}_{I}(\mathcal{R})_{\mathbf{0}}\\
&= \left(\mathcal{M}_{I}(\mathcal{R})_{g^{\prime}} \star \mathcal{M}_{I}(\mathcal{R})_{-\beta^{-1}\circ\alpha(g^{\prime})}\right) \star \phi\left(\phi^{-1}\left(\mathcal{M}_{I}(\mathcal{R})_{\mathbf{0}}\right)\right)\\
&= \psi\left(\mathcal{M}_{I}(\mathcal{R})_{g^{\prime}}\right) \star\left(\mathcal{M}_{I}(\mathcal{R})_{-\beta^{-1}\circ\alpha(g^{\prime})} \star \phi^{-1}\left(\mathcal{M}_{I}(\mathcal{R})_{\mathbf{0}}\right)\right)\\
&\subset \mathcal{M}_{I}(\mathcal{R})_{\alpha\left(g^{\prime}\right)} \star \mathcal{M}_{I}(\mathcal{R})_{-\beta^{-1}\circ\alpha^2\left(g^{\prime}\right)} \subset \mathfrak{I}_{\mathbf{0},[g]},
\end{align*}
and then
\begin{align}\label{eq:ideal-right-zero-from-I0}
\mathfrak{I}_{\mathbf{0},[g]} \star \mathcal{M}_{I}(\mathcal{R})_{\mathbf{0}} \subseteq \mathfrak{I}_{\mathbf{0},[g]}.
\end{align}
From equations~\eqref{eq:ideal-right-zero-from-V} and~\eqref{eq:ideal-right-zero-from-I0}, we obtain
$$
\mathfrak{I}_{[g]} \star \mathcal{M}_{I}(\mathcal{R})_{\mathbf{0}} \subset \mathfrak{I}_{[g]}
$$
Taking into account the above observation and Lemma~\ref{lemma_subideal}, we have
$$
\mathfrak{I}_{[g]} \star \mathcal{M}_{I}(\mathcal{R})=\mathfrak{I}_{[g]} \star\left(\mathcal{M}_{I}(\mathcal{R})_{\mathbf{0}} \oplus\left(\bigoplus_{g^{\prime} \in[g]} \mathcal{M}_{I}(\mathcal{R})_{g^{\prime}}\right) \oplus\left(\bigoplus_{g^{\prime \prime} \notin[g]} \mathcal{M}_{I}(\mathcal{R})_{g^{\prime \prime}}\right)\right) \subset \mathfrak{I}_{[g]}
$$
In a similar way we get $\mathcal{M}_{I}(\mathcal{R}) \star \mathfrak{I}_{[g]} \subset \mathfrak{I}_{[g]}$. Thus we conclude $\mathfrak{I}_{[g]}$ is a graded ideal of $\mathcal{M}_{I}(\mathcal{R})$.
\end{proof}

\begin{theorem}\label{B'}
If the BiHom-algebra $\left( \mathcal{M}_{I}(\mathcal{R}), \star, \psi, \phi \right)$ is graded simple, then the following hold:

\begin{itemize}
    \item For every pair $g, g' \in \Sigma$, there exists a connection from $g$ to $g'$.
    
    \item The degree-zero component $\mathcal{M}_{I}(\mathcal{R})_{\mathbf{0}}$ can be written as
    \[
    \mathcal{M}_{I}(\mathcal{R})_{\mathbf{0}} 
    = \sum_{g \in \Sigma} \left( \mathcal{M}_{I}(\mathcal{R})_g \star \mathcal{M}_{I}(\mathcal{R})_{-\beta^{-1} \circ \alpha(g)} \right).
    \]
\end{itemize}
\end{theorem}

\begin{proof}
The graded simplicity of the BiHom-algebra $\left( \mathcal{M}_{I}(\mathcal{R}), \star, \psi, \phi \right)$ implies that for any $g \in \Sigma$, the associated graded ideal satisfies
\[
\mathfrak{I}_{[g]} = \mathcal{M}_{I}(\mathcal{R}).
\]
In particular, this means that the sum of all homogeneous components indexed by $[g]$ must exhaust the entire grading support $\Sigma$. Hence, we conclude that $[g] = \Sigma$; that is, every pair $g, g' \in \Sigma$ lies in the same equivalence class under the connection relation.
\end{proof}

\begin{theorem}
We have the decomposition
\[
\mathcal{M}_{I}(\mathcal{R}) = U + \sum_{[g] \in \Sigma / \sim} \mathfrak{I}_{[g]},
\]
where $U$ is an $\mathcal{R}$-submodule of the degree-zero component $\mathcal{M}_{I}(\mathcal{R})_{\mathbf{0}}$, and each $\mathfrak{I}_{[g]}$ is one of the graded ideals of the BiHom-algebra $\left( \mathcal{M}_{I}(\mathcal{R}), \star, \psi, \phi \right)$ as described in Theorem~\ref{lemma_idealofgraded}.
\end{theorem}

\begin{proof}
By Theorem~\ref{lemma_idealofgraded}, each $\mathfrak{I}_{[g]}$ is a graded ideal of $(\mathcal{M}_{I}(\mathcal{R}), \star, \psi, \phi)$. Consider an $\mathcal{R}$-submodule $U \subseteq \mathcal{M}_{I}(\mathcal{R})_{\mathbf{0}}$ such that
\[
\mathcal{M}_{I}(\mathcal{R})_{\mathbf{0}} = U + \operatorname{Span}_{\mathcal{R}} 
\left\{ \mathcal{M}_{I}(\mathcal{R})_g \star \mathcal{M}_{I}(\mathcal{R})_{-\beta^{-1} \circ \alpha(g)} 
\mid g \in \Sigma \right\}.
\]

Now, applying Lemma~\ref{lem:orthogonal-classes}, and using the fact that
\[
\mathcal{M}_{I}(\mathcal{R}) = \mathcal{M}_{I}(\mathcal{R})_{\mathbf{0}} 
\oplus 
\left( \bigoplus_{g \in \Sigma} \mathcal{M}_{I}(\mathcal{R})_g \right),
\]
we obtain the desired decomposition:
\[
\mathcal{M}_{I}(\mathcal{R}) = U + \sum_{[g] \in \Sigma / \sim} \mathfrak{I}_{[g]}.
\]
\end{proof}

Let us denote by 
\[
\mathbb{Z}\left(\mathcal{M}_{I}(\mathcal{R})\right) 
:= \left\{ v \in \mathcal{M}_{I}(\mathcal{R}) \mid v \star \mathcal{M}_{I}(\mathcal{R}) 
+ \mathcal{M}_{I}(\mathcal{R}) \star v = \mathbf{0} \right\}
\]
the centre of $\mathcal{M}_{I}(\mathcal{R})$.

\begin{corollary}\label{corollary_theorem}
If $\mathbb{Z}\left(\mathcal{M}_{I}(\mathcal{R})\right) = \mathbf{0}$, and if
\[
\mathcal{M}_{I}(\mathcal{R})_{\mathbf{0}} 
= \sum_{g \in \Sigma} \left( \mathcal{M}_{I}(\mathcal{R})_g 
\star 
\mathcal{M}_{I}(\mathcal{R})_{-\beta^{-1} \circ \alpha(g)} \right),
\]
then the BiHom-algebra $\left( \mathcal{M}_{I}(\mathcal{R}), \star, \psi, \phi \right)$ decomposes as a direct sum of the graded ideals given in Theorem~\ref{lemma_idealofgraded}, that is,
\[
\mathcal{M}_{I}(\mathcal{R}) = \bigoplus_{[g] \in \Sigma / \sim} \mathfrak{I}_{[g]}.
\]
\end{corollary}

\begin{proof}
From the identity
\[
\mathcal{M}_{I}(\mathcal{R})_{\mathbf{0}} 
= \sum_{g \in \Sigma} \left( 
\mathcal{M}_{I}(\mathcal{R})_g 
\star 
\mathcal{M}_{I}(\mathcal{R})_{-\beta^{-1} \circ \alpha(g)} \right),
\]
it follows that
\[
\mathcal{M}_{I}(\mathcal{R}) 
= \sum_{[g] \in \Sigma / \sim} \mathfrak{I}_{[g]}.
\]

To show that the sum is direct, we recall from Lemma~\ref{lem:orthogonal-classes} that
\[
\mathfrak{I}_{[g]} \star \mathfrak{I}_{[g']}=\mathbf{0}
\quad\text{whenever }[g]\neq[g'].
\]
In addition, since the centre 
$\mathbb{Z}\left( \mathcal{M}_{I}(\mathcal{R}) \right) = \mathbf{0}$, 
this implies that the intersection between distinct graded ideals is trivial. 

Hence, the decomposition is direct:
\[
\mathcal{M}_{I}(\mathcal{R}) = \bigoplus_{[g] \in \Sigma / \sim} \mathfrak{I}_{[g]}.
\]
\end{proof}

\section{Graded simple case}\label{sect5}

In this section, we consider an arbitrary graded BiHom-algebra over a general base field. Our objective is to investigate the conditions under which such an algebra admits a decomposition as the direct sum of its simple graded ideals.

From this point onwards, we let $(\mathcal{M}, \star, \psi, \phi)$ denote a BiHom-algebra of arbitrary (possibly infinite) dimension over an arbitrary base field $\mathbb{K}$, graded by the regular abelian BiHom-group  $ \left( G, +_{(\alpha, \beta)}, \alpha, \beta, 0 \right)$,
where the operation $+_{(\alpha,\beta)}$ is defined by $g +_{(\alpha,\beta)} g' := \alpha(g) + \beta(g')$, and $\alpha, \beta$ are two commuting automorphisms of the underlying abelian group $(G, +, 0)$.

Accordingly, we write the decomposition of $\mathcal{M}$ as
\[
\mathcal{M} = \bigoplus_{g \in G} \mathcal{M}_g = \mathcal{M}_0 \oplus \left( \bigoplus_{g \in \Sigma} \mathcal{M}_g \right),
\]
where the multiplication satisfies
\begin{align}\label{4.1}
\mathcal{M}_g \star \mathcal{M}_{g'} \subseteq \mathcal{M}_{g +_{(\alpha, \beta)} g'} 
= \mathcal{M}_{\alpha(g) + \beta(g')}, \quad \text{for all } g, g' \in G,
\end{align}
and where the support of the grading is given by
\[
\Sigma := \left\{ g \in G \setminus \{0\} \mid \mathcal{M}_g \neq 0 \right\}.
\]

Furthermore, we assume that the twisting maps $\psi,\phi:\mathcal{M}\to\mathcal{M}$, in the sense of Definition~\ref{def:BiHomAlg}, are commuting graded automorphisms of $(\mathcal{M},\star)$, compatible with $\alpha$ and $\beta$. That is,
\begin{align}
\psi\left( \mathcal{M}_g \right) &\subseteq \mathcal{M}_{\alpha(g)}, \label{4.2} \\
\phi\left( \mathcal{M}_g \right) &\subseteq \mathcal{M}_{\beta(g)}, \label{4.3}
\end{align}
for all $g \in G$.

Observe that if $\alpha = \beta$ (respectively, $\alpha = \beta = \operatorname{Id}_G$) and $\psi = \phi$ (respectively, $\psi = \phi = \operatorname{Id}_{\mathcal{M}}$), then the structure reduces to a Hom-algebra (respectively, an associative algebra) graded by an abelian Hom-group (respectively, by a classical abelian group). Therefore, the results presented in this section extend those obtained in~\cite{C2}.

We recall (see Sect.~\ref{sect2}) that a graded ideal of $\mathcal{M}$ is a linear subspace $\mathfrak{I} \subseteq \mathcal{M}$ satisfying the following conditions:
\begin{itemize}
    \item $\mathfrak{I} \star \mathcal{M} + \mathcal{M} \star \mathfrak{I} \subseteq \mathfrak{I}$,
    \item $\psi(\mathfrak{I}) = \mathfrak{I}$ and $\phi(\mathfrak{I}) = \mathfrak{I}$,
    \item $\mathfrak{I}$ admits a grading: $\mathfrak{I} = \bigoplus_{g \in G} \mathfrak{I}_g$, where $\mathfrak{I}_g = \mathfrak{I} \cap \mathcal{M}_g$.
\end{itemize}
We also recall that the BiHom-algebra $\mathcal{M}$ is said to be \textit{graded simple} if $\mathcal{M} \star \mathcal{M} \neq \mathbf{0}$, and the only graded ideals of $\mathcal{M}$ are $\{0\}$ and $\mathcal{M}$ itself.

As usual, we denote by
\[
\mathbb{Z}(\mathcal{M}) := \left\{ v \in \mathcal{M} \mid v \star \mathcal{M} + \mathcal{M} \star v = \mathbf{0} \right\}
\]
the centre of the BiHom-algebra $(\mathcal{M}, \star, \psi, \phi)$.

\begin{lemma}\label{lemma_center}
Suppose that
\[
\mathcal{M}_{\mathbf{0}} 
= \sum_{g \in \Sigma} 
\mathcal{M}_g \star \mathcal{M}_{-\beta^{-1} \circ \alpha(g)}.
\]
If $\mathfrak{I}$ is a graded ideal of the BiHom-algebra $(\mathcal{M}, \star, \psi, \phi)$ such that $\mathfrak{I} \subseteq \mathcal{M}_{\mathbf{0}}$, then $\mathfrak{I} \subseteq \mathbb{Z}(\mathcal{M})$.
\end{lemma}

\begin{proof}
Since $\mathfrak{I}$ is a graded ideal of the BiHom-algebra $(\mathcal{M}, \star, \psi, \phi)$ with $\mathfrak{I} \subseteq \mathcal{M}_{\mathbf{0}}$, we have
\[
\mathfrak{I} \star \left( \bigoplus_{g \in \Sigma} \mathcal{M}_g \right) \subseteq \mathfrak{I} \subseteq \mathcal{M}_{\mathbf{0}}.
\]
On the other hand, using equation~\eqref{4.1}, it follows that
\[
\mathfrak{I} \star \left( \bigoplus_{g \in \Sigma} \mathcal{M}_g \right)
\subseteq 
\mathcal{M}_{\mathbf{0}} \star \left( \bigoplus_{g \in \Sigma} \mathcal{M}_g \right)
\subseteq 
\bigoplus_{g \in \Sigma} \mathcal{M}_{\beta(g)} 
\subseteq 
\bigoplus_{g \in \Sigma} \mathcal{M}_g.
\]
Thus,
\[
\mathfrak{I} \star \left( \bigoplus_{g \in \Sigma} \mathcal{M}_g \right) 
\subseteq 
\left( \bigoplus_{g \in \Sigma} \mathcal{M}_g \right) \cap \mathcal{M}_{\mathbf{0}} = \mathbf{0},
\]
and hence,
\[
\mathfrak{I} \star \left( \bigoplus_{g \in \Sigma} \mathcal{M}_g \right) = \mathbf{0}.
\]
Similarly, one obtains:
\[
\left( \bigoplus_{g \in \Sigma} \mathcal{M}_g \right) \star \mathfrak{I} = \mathbf{0}.
\]
Therefore,
\begin{equation}\label{4.4}
\mathfrak{I} \star \left( \bigoplus_{g \in \Sigma} \mathcal{M}_g \right) 
+ 
\left( \bigoplus_{g \in \Sigma} \mathcal{M}_g \right) \star \mathfrak{I} 
= \mathbf{0}.
\end{equation}

Now, using BiHom-associativity and equation~\eqref{4.4}, for each $g \in \Sigma$, we compute:
\begin{align*}
\mathfrak{I} \star 
\left( 
\mathcal{M}_g \star \mathcal{M}_{-\beta^{-1} \circ \alpha(g)} 
\right) 
&= \psi \left( \psi^{-1}(\mathfrak{I}) \right) 
\star 
\left( \mathcal{M}_g \star \mathcal{M}_{-\beta^{-1} \circ \alpha(g)} \right) \\
&= \left( \psi^{-1}(\mathfrak{I}) \star \mathcal{M}_g \right) 
\star 
\phi \left( \mathcal{M}_{-\beta^{-1} \circ \alpha(g)} \right) \\
&= \mathbf{0}.
\end{align*}

Similarly,
\[
\left( \mathcal{M}_g \star \mathcal{M}_{-\beta^{-1} \circ \alpha(g)} \right) \star \mathfrak{I} = \mathbf{0}.
\]

Since
\[
\mathcal{M}_{\mathbf{0}} 
= \sum_{g \in \Sigma} \left( \mathcal{M}_g \star \mathcal{M}_{-\beta^{-1} \circ \alpha(g)} \right),
\]
we conclude that
\begin{equation}\label{4.5}
\mathfrak{I} \star \mathcal{M}_{\mathbf{0}} 
+ 
\mathcal{M}_{\mathbf{0}} \star \mathfrak{I} 
= \mathbf{0}.
\end{equation}

Combining equations~\eqref{4.4} and~\eqref{4.5}, we obtain:
\[
\mathfrak{I} \star \mathcal{M} + \mathcal{M} \star \mathfrak{I} = \mathbf{0},
\]
and thus $\mathfrak{I} \subseteq \mathbb{Z}(\mathcal{M})$, as required.
\end{proof}

We now introduce the notions of {\em $\Sigma$-multiplicativity} and {\em maximal length} within the setting of graded BiHom-algebras. These concepts are defined in a manner analogous to their counterparts in the theory of graded associative algebras, graded Lie algebras, split Lie colour algebras, and split Leibniz algebras (see~\cite{C1,C2,CS,CS2} for definitions and illustrative examples).

\begin{definition}
Let $(\mathcal{M}, \star, \psi, \phi)$ be a BiHom-algebra graded by the regular BiHom-group $\left(G, +_{(\alpha,\beta)}, \alpha, \beta, \mathbf{0} \right)$. We say that $\mathcal{M}$ is \textbf{$\Sigma$-multiplicative} if for all $g, g' \in \Sigma$ such that $g +_{(\alpha, \beta)} g' \in \Sigma \cup \{\mathbf{0}\}$, the following equality holds:
\[
\mathcal{M}_g \star \mathcal{M}_{g'} = \mathcal{M}_{g +_{(\alpha, \beta)} g'}.
\]
\end{definition}

We say that a graded BiHom-algebra $(\mathcal{M}, \star, \psi, \phi)$ is of \textbf{maximal length} if $\operatorname{dim} \mathcal{M}_g = 1$ for every $g \in \Sigma$.

\begin{remark}
    It is worth noting that the notions of $\Sigma$-multiplicativity and maximal length arise naturally in the study of gradings on algebras. In particular, certain structural and computational problems can often be formulated more naturally — and with greater simplicity — when one works with gradings that are both $\Sigma$-multiplicative and of maximal length.
\end{remark}

To illustrate this point, let us consider the specific case of the matrix algebra $\mathcal{M}_n(\mathbb{K})$, where gradings induced by the Pauli matrices have played a fundamental role in the context of mathematical physics. These gradings are both $\Sigma$-multiplicative and of maximal length, and this structural property is essential in their applicability.

The key reason is that such gradings provide a decomposition of $\mathcal{M}_n(\mathbb{K})$ into a direct sum of $n^2$ one-dimensional subspaces—effectively yielding a basis for $\mathcal{M}_n(\mathbb{K})$—whose generators are all semisimple elements of the algebra.

Let us now describe these gradings in detail.
\begin{example}
  Consider the matrix algebra $\mathcal{M} = \mathcal{M}_2(\mathbb{C})$. A $\mathbb{Z}_2 \times \mathbb{Z}_2$-grading on $\mathcal{M}$ can be defined using the Pauli matrices:
\[
\sigma_1 = \begin{pmatrix}
0 & 1 \\
1 & 0
\end{pmatrix}, \quad
\sigma_2 = \begin{pmatrix}
0 & i \\
-i & 0
\end{pmatrix}, \quad
\sigma_3 = \begin{pmatrix}
1 & 0 \\
0 & -1
\end{pmatrix}.
\]

This grading is given explicitly by:
\[
\begin{array}{ll}
\mathcal{M}_{(0,0)} = \left\{ \begin{pmatrix}
a & 0 \\
0 & a
\end{pmatrix} \;\middle|\; a \in \mathbb{C} \right\}, &
\mathcal{M}_{(1,0)} = \left\{ \begin{pmatrix}
b & 0 \\
0 & -b
\end{pmatrix} \;\middle|\; b \in \mathbb{C} \right\}, \\[1.5em]
\mathcal{M}_{(0,1)} = \left\{ \begin{pmatrix}
0 & c \\
c & 0
\end{pmatrix} \;\middle|\; c \in \mathbb{C} \right\}, &
\mathcal{M}_{(1,1)} = \left\{ \begin{pmatrix}
0 & d \\
-d & 0
\end{pmatrix} \;\middle|\; d \in \mathbb{C} \right\}.
\end{array}
\]
Each homogeneous component is one-dimensional and spanned by a Pauli matrix (or scalar multiple thereof), making this grading both $\Sigma$-multiplicative and of maximal length.

If $\mathbb{K}$ contains a primitive $n$-th root of unity, then the Pauli matrices $-\sigma_{3}$ and $\sigma_{1}$ can be generalized to

$$
X_{a}=\left(\begin{array}{cccccc}
\epsilon^{n-1} & 0 & 0 & \cdots & 0 & 0 \\
0 & \epsilon^{n-2} & 0 & \cdots & 0 & 0 \\
\cdots & & & & & \\
0 & 0 & 0 & \cdots & \epsilon & 0 \\
0 & 0 & 0 & \cdots & 0 & 1
\end{array}\right) \text { and } X_{b}=\left(\begin{array}{cccccc}
0 & 1 & 0 & \cdots & 0 & 0 \\
0 & 0 & 1 & \cdots & 0 & 0 \\
\cdots & & & & & \\
0 & 0 & 0 & \cdots & 0 & 1 \\
1 & 0 & 0 & \cdots & 0 & 0
\end{array}\right)
$$

Let $\mathcal{M} = \mathcal{M}_n(\mathbb{K})$, and consider the elements $X_a, X_b \in \mathcal{M}$ satisfying the relations:
\[
X_a X_b = \epsilon X_b X_a, \quad X_a^n = X_b^n = I_n,
\]
for some fixed primitive $n$th root of unity $\epsilon \in \mathbb{K}$. These generate a $\mathbb{Z}_n \times \mathbb{Z}_n$-grading on $\mathcal{M}$ given by:
\[
\mathcal{M}_{(k, \ell)} = \mathbb{K} \cdot X_a^k X_b^\ell, \quad \text{for all } (k, \ell) \in \mathbb{Z}_n \times \mathbb{Z}_n.
\]
It is straightforward to verify that this grading is both $\Sigma$-multiplicative and of maximal length. In classical literature, the term \emph{fine} is used to refer to any grading on $\mathcal{M}_n(\mathbb{K})$ for which $\dim \mathcal{M}_g \leq 1$ for all $g \in G$. The importance of such gradings is clearly justified therein, and in particular, every grading of maximal length is fine in this sense. Moreover, it was known that (for fields of characteristic zero, and later extended to arbitrary characteristic) that any grading on $\mathcal{M}_n(\mathbb{K})$ can be obtained by combining elementary gradings with gradings that are both $\Sigma$-multiplicative and of maximal length.
\end{example}
 
\begin{remark}\label{remark_same}
We observe that the definitions, arguments, and results established in Sections~2 and~3 for graded BiHom-algebras of the form $(\mathcal{M}_{I}(\mathcal{R}), \star, \psi, \phi)$ remain valid in the more general context of graded BiHom-algebras $(\mathcal{M}, \star, \psi, \phi)$ considered in this section.
\end{remark}

\begin{theorem}\label{graded_simple_judgement}
Let $(\mathcal{M}, \star, \psi, \phi)$ be a graded BiHom-algebra that is both of maximal length and $\Sigma$-multiplicative. Then $(\mathcal{M}, \star, \psi, \phi)$ is graded simple if and only if the following conditions are satisfied:
\begin{enumerate}
    \item $\mathbb{Z}(\mathcal{M}) = \mathbf{0}$;
    \item $\mathcal{M}_{\mathbf{0}} = \sum_{g \in \Sigma} \left( \mathcal{M}_g \star \mathcal{M}_{-\beta^{-1} \circ \alpha(g)} \right)$;
    \item All elements of $\Sigma$ are connected.
\end{enumerate}
\end{theorem}

\begin{proof}
    Suppose the BiHom-algebra $(\mathcal{M}, \star,\psi,\phi)$ is graded simple. It is easy to verify that $\mathbb{Z}(\mathcal{M})$ is a graded ideal of $\mathcal{M}$ and so $\mathbb{Z}(\mathcal{M})=0$. By Theorem~\ref{B'} and Remark~\ref{remark_same}, we complete the proof of the first implication.
    
    To prove the converse, we consider $\mathfrak{I}=\bigoplus_{g \in G} \mathfrak{I}_{g}$ where $\mathfrak{I}_{g}=\mathfrak{I} \cap \mathcal{M}_{g}$, a nonzero graded ideal of the BiHom-algebra $(\mathcal{M}, \star,\psi,\phi)$. By the maximal length of the BiHom-algebra $(\mathcal{M}, \star,\psi,\phi)$, if we denote by $\Sigma_{\mathfrak{I}}:=\left\{g \in \Sigma~|~\mathfrak{I}_{g} \neq 0\right\}$, we can write $\mathfrak{I}=\mathfrak{I}_{0} \oplus\left(\bigoplus_{g \in \Sigma_{\mathfrak{I}}} \mathcal{M}_{g}\right)$, being also $\Sigma_{\mathfrak{I}} \neq \emptyset$ as consequence of Lemma~\ref{lemma_center}. Hence, we can take $g_{0} \in \Sigma_{\mathfrak{I}}$ being so $0 \neq \mathcal{M}_{g_{0}} \subset \mathfrak{I}$. Since the fact $\psi(\mathfrak{I})=\mathfrak{I},\phi(\mathfrak{I})=\mathfrak{I}$ and two equations~\eqref{4.2} and ~\eqref{4.3} give us that
\begin{align}\label{4.6}
    \text { if } g \in \Sigma_{\mathfrak{I}} \text { then }\left\{\alpha^{i}\circ\beta^j(g)~|~i,j \in \mathbb{Z}\right\} \subset \Sigma_{\mathfrak{I}}.
\end{align}
For $-g\in\Sigma$, the $\Sigma$-multiplicativity gives us that
\begin{align*}
    \mathcal{M}_{-g}=\mathcal{M}_{g}\star\mathcal{M}_{-\beta^{-1}\circ\alpha(g)-\beta^{-1}(g)}\subset\mathfrak{I}\star\mathcal{M}_{-\beta^{-1}\circ\alpha(g)-\beta^{-1}(g)}\subset\mathfrak{I}.
\end{align*}
From the equation~\eqref{4.6}, we can obtain that $\left\{-\alpha^{i}\circ\beta^j(g)~|~i,j \in \mathbb{Z}\right\} \subset \Sigma_{\mathfrak{I}}$, and we get
\begin{align}\label{4.7}
    \left\{\mathcal{M}_{\pm \alpha^{i}\circ\beta^j\left(g_{0}\right)}~|~i,j\in \mathbb{Z}\right\} \subset \mathfrak{I}.
\end{align}

Now consider $g^{\prime} \in \Sigma$ such that $g^{\prime} \notin\left\{ \pm \alpha^{i}\circ\beta^j\left(g_{0}\right)~|~i,j\in \mathbb{Z}\right\}$. The fact that $g_{0}$ and $g^{\prime}$ are connected gives us a connection $\left\{g_{1}, g_{2}, \ldots, g_{k}\right\}, k \geq 2$, from $g_{0}$ to $g^{\prime}$ such that
\begin{align*}
& g_{1}=\alpha^{i}\circ\beta^j\left(g_{0}\right) \text { for some } i,j \in \mathbb{N} . \\
& \alpha\left(g_{1}\right)+\beta\left(g_{2}\right) \in \Sigma \\
& \cdots \cdots \cdots \\
& \alpha^{i}\left(g_{1}\right)+\alpha^{i-1}\circ\beta\left(g_{2}\right)+\alpha^{i-2}\circ\beta\left(g_{3}\right)+\cdots+\beta\left(g_{i+1}\right) \in \Sigma \\
& \cdots \cdots \cdots \\
& \alpha^{k-2}\left(g_{1}\right)+\alpha^{k-3}\circ\beta\left(g_{2}\right)+\alpha^{k-4}\circ\beta\left(g_{3}\right)+\cdots+\alpha^{k-i-1}\circ\beta\left(g_{i}\right)+\cdots+\beta\left(g_{k-1}\right) \in \Sigma \\
& \alpha^{k-1}\left(g_{1}\right)+\alpha^{k-2}\circ\beta\left(g_{2}\right)+\alpha^{k-3}\circ\beta\left(g_{3}\right)+\cdots+\alpha^{k-i}\circ\beta\left(g_{i}\right)+\cdots+\beta\left(g_{k}\right)=\epsilon \alpha^{m}\circ\beta^n\left(g^{\prime}\right)
\end{align*}
$\text { for some } m,n \in \mathbb{N} \text { and } \epsilon \in \{\pm 1\}$. Denote $g_{1}, g_{2} \in \Sigma$ and the $\Sigma$-multiplicativity gives us that $g_{1}+{ }_{(\alpha,\beta)} g_{2}=\alpha\left(g_{1}\right)+\beta\left(g_{2}\right) \in \Sigma$. Then $\mathcal{M}_{g_{1}} \star \mathcal{M}_{g_{2}}=\mathcal{M}_{\alpha\left(g_{1}\right)+\beta\left(g_{2}\right)}\neq 0$. Now taking into account that $\mathcal{M}_{g_{1}} \subset \mathfrak{I}$ as consequence of equation~\eqref{4.7}, we get
$$
0 \neq \mathcal{M}_{\alpha\left(g_{1}\right)+\beta\left(g_{2}\right)}=\mathcal{M}_{g_{1}} \star \mathcal{M}_{g_{2}}\subset\mathfrak{I}\star\mathcal{M}_{g_{2}} \subset \mathfrak{I}.
$$
We can argue in a similar way from $\alpha\left(g_{1}\right)+\beta\left(g_{2}\right), g_{3}$ and $\left(\alpha\left(g_{1}\right)+\beta\left(g_{2}\right)\right)+{ }_{(\alpha,\beta)} g_{3}=\alpha^{2}\left(g_{1}\right)+\alpha\circ\beta\left(g_{2}\right)+\beta\left(g_{3}\right)$ to get
\[
0\neq
\mathcal{M}_{\alpha^{2}(g_1)+\alpha\circ\beta(g_2)+\beta(g_3)}
\subseteq \mathfrak{I}.
\]
Following this process with the connection $\left\{g_{1}, g_{2}, \ldots, g_{k}\right\}$ we obtain that
$$
0 \neq \mathcal{M}_{\alpha^{k-1}\left(g_{1}\right)+\alpha^{k-2}\circ\beta\left(g_{2}\right)+\alpha^{k-3}\circ\beta\left(g_{3}\right)+\cdots+\alpha^{k-i}\circ\beta\left(g_{i}\right)+\cdots+\beta\left(g_{k}\right)} \subset \mathfrak{I}
$$
and so either $\mathcal{M}_{\alpha^{m}\circ\beta^n\left(g^{\prime}\right)} \subset \mathfrak{I}$ or $\mathcal{M}_{-\alpha^{m}\circ\beta^n\left(g^{\prime}\right)} \subset \mathfrak{I}$. Now taking into account two equations~\eqref{4.6} and~\eqref{4.7}, we get
\begin{align}\label{4.8}
\text {either }\left\{\mathcal{M}_{\alpha^{m}\circ\beta^n(g)}~|~m,n \in \mathbb{Z}\right\} \subset \mathfrak{I} \text { or }\left\{\mathcal{M}_{-\alpha^{m}\circ\beta^n(g)}~|~m,n \in \mathbb{Z}\right\} \subset \mathfrak{I} \text { for each } g \in \Sigma.
\end{align}
Observe that equation~\eqref{4.8} can be reformulated by stating that given any $g \in \Sigma$ either $\left\{\alpha^{m}\circ\beta^n(g)~|~m,n\in \mathbb{Z}\right\}$ or $\left\{\alpha^{m}\circ\beta^n(-g)~|~m,n \in \mathbb{Z}\right\}$ is contained in $\Sigma_{\mathfrak{I}}$. Since $\mathcal{M}_{\mathbf{0}}=\sum_{g \in \Sigma}\left(\mathcal{M}_{g} \star \mathcal{M}_{-\beta^{-1}\circ\alpha(g)}\right)$ we have that 
$$
\mathcal{M}_{\mathbf{0}}\subset\sum_{g \in \Sigma}\left(\mathfrak{I} \star \mathcal{M}_{-\beta^{-1}\circ\alpha(g)}\right)\subset\mathfrak{I},
$$
or
$$
\mathcal{M}_{\mathbf{0}}\subset\sum_{g \in \Sigma}\left(\mathcal{M}_{g} \star \mathfrak{I}\right)\subset\mathfrak{I}.
$$
Thus, we get
\begin{align}\label{4.9}
    \mathcal{M}_{\mathbf{0}} \subset \mathfrak{I}.
\end{align}

Now we will prove that $\Sigma_{\mathfrak{I}}=\Sigma$. Using the equation~\eqref{4.9}, for any $g\in\Sigma$, we have 
\begin{align*}
    \mathcal{M}_g=\mathcal{M}_{\mathbf{0}}\star \mathcal{M}_{\beta^{-1}(g)}\subset \mathfrak{I}\star \mathcal{M}_{\beta^{-1}(g)}\subset \mathfrak{I}.
\end{align*}
Thus $-g \in \Sigma_{\mathfrak{I}}$ and $\Sigma_{\mathfrak{I}}=\Sigma$.
Therefore, by the equation~\eqref{4.9}, the nonzero graded ideal $\mathfrak{I}$ can be represented as
$$
\mathfrak{I}=\mathfrak{I}_{0} \oplus\left(\bigoplus_{g \in \Sigma_{\mathfrak{I}}} \mathcal{M}_{g}\right)=\mathcal{M}_{\mathbf{0}}\oplus\left(\bigoplus_{g \in \Sigma} \mathcal{M}_{g}\right)=\mathcal{M}.
$$
Hence $\mathcal{M}$ is a graded simple.
\end{proof}

\begin{theorem}
Let $(\mathcal{M}, \star, \psi, \phi)$ be a graded BiHom-algebra that is of maximal length and $\Sigma$-multiplicative, with trivial centre $\mathbb{Z}(\mathcal{M}) = \mathbf{0}$, and satisfying \[
\mathcal{M}_{\mathbf{0}} = \sum_{g \in \Sigma} \left( \mathcal{M}_g \star \mathcal{M}_{-\beta^{-1} \circ \alpha(g)} \right).
\]
Then the algebra decomposes as a direct sum of graded ideals:
\[
\mathcal{M} = \bigoplus_{[g] \in \Sigma / \sim} \mathfrak{I}_
{[g]},
\]
where each graded ideal $\mathfrak{I}_{[g]}$ is graded simple, and its support $\Sigma_{\mathfrak{I}_{[g]}}$ consists entirely of elements that are pairwise $\Sigma_{\mathfrak{I}_{[g]}}$-connected.
\end{theorem}

\begin{proof}
By Corollary~\ref{corollary_theorem} and Remark~\ref{remark_same}, we have the direct sum decomposition
\[
\mathcal{M} = \bigoplus_{[g] \in \Sigma / \sim} \mathfrak{I}_{[g]},
\]
where each ideal \(\mathfrak{I}_{[g]}\) is given by
\[
\mathfrak{I}_{[g]}
=
\mathfrak{I}_{\mathbf{0},[g]} \oplus V_{[g]}
=
\operatorname{Span}_{\mathbb{K}}
\left\{
\mathcal{M}_{g'} \star
\mathcal{M}_{-\beta^{-1}\circ\alpha(g')}
\mid g' \in [g]
\right\}
\oplus
\bigoplus_{g' \in [g]} \mathcal{M}_{g'}.
\]

The support of $\mathfrak{I}_{[g]}$ is precisely $\Sigma_{\mathfrak{I}_{[g]}} = [g]$, and by construction, all elements of $\Sigma_{\mathfrak{I}_{[g]}}$ are pairwise $\Sigma_{\mathfrak{I}_{[g]}}$-connected.

Moreover, since $\mathcal{M}$ is $\Sigma$-multiplicative and of maximal length, the same properties are inherited by each $\mathfrak{I}_{[g]}$. From Lemma~\ref{lem:orthogonal-classes} and the assumption that $\mathbb{Z}(\mathcal{M}) = \mathbf{0}$, it follows that the centre of each $\mathfrak{I}_{[g]}$ in itself is trivial, i.e.,
\[
\mathbb{Z}(\mathfrak{I}_{[g]}) = \mathbf{0}.
\]

We may therefore apply Theorem~\ref{graded_simple_judgement} to each graded ideal $\mathfrak{I}_{[g]}$, concluding that every $\mathfrak{I}_{[g]}$ is graded simple. 
\end{proof}












\section*{Acknowledgements}
The authors would like to thank the referee for the careful reading and helpful suggestions, which have improved the manuscript.
S.H. Wang was partially supported by the National Natural Science Foundation of China (Grant No. 12271089). H. Zhu gratefully acknowledges SPMS at NTU-SG for its warm academic
environment and the support of the Research Scholarship (NTU-RSS).

\bibliographystyle{plain}

\begin{thebibliography}{99}

\bibitem{AHM}
K. Abdaoui, A. Ben Hassine and A. Makhlouf,
\emph{BiHom-Lie colour algebras structures},
preprint, arXiv:1706.02188, 2017.

\bibitem{EDN}
E. Aljadeff, D. Haile and M. Natapov,
\emph{Graded identities of matrix algebras and the universal graded algebra},
Trans. Amer. Math. Soc. \textbf{362} (2010), no.~6, 3125--3147.

\bibitem{AragonCalderon2012}
M.J. Arag\'on and A.J. Calder\'on,
\emph{On graded matrix Hom-algebras},
Electron. J. Linear Algebra \textbf{24} (2012), 45--65.

\bibitem{BZ}
Y. Bahturin and M.V. Zaicev,
\emph{Involutions on graded matrix algebras},
J. Algebra \textbf{315} (2007), no.~2, 527--540.

\bibitem{BarreiroCalderonNavarroSanchez2025}
E. Barreiro, A.J. Calder\'on, R.M. Navarro and J.M. S\'anchez,
\emph{Graded regular BiHom-Lie algebras},
Algebra Colloq. \textbf{32} (2025), no.~3, 429--442.

\bibitem{HMN}
A. Ben Hassine, S. Mabrouk and O. Ncib,
\emph{$3$-BiHom-Lie superalgebras induced by BiHom-Lie superalgebras},
Linear Multilinear Algebra \textbf{70} (2022), no.~1, 101--121.

\bibitem{CalderonMartin2008}
A.J. Calder\'on Mart\'in,
\emph{On split Lie algebras with symmetric root systems},
Proc. Indian Acad. Sci. Math. Sci. \textbf{118} (2008), no.~3, 351--356.

\bibitem{C1}
A.J. Calder\'on,
\emph{On the structure of graded Lie algebras},
J. Math. Phys. \textbf{50} (2009), 103513, 8 pp.

\bibitem{C2}
A.J. Calder\'on,
\emph{On graded associative algebras},
Rep. Math. Phys. \textbf{69} (2012), no.~1, 75--86.

\bibitem{CDM}
A.J. Calder\'on, C. Draper and C. Mart\'in,
\emph{Gradings on the real forms of the Albert algebra, of $\mathfrak{g}_2$,
and of $\mathfrak{f}_4$},
J. Math. Phys. \textbf{51} (2010), 053516.

\bibitem{CS}
A.J. Calder\'on and J.M. S\'anchez,
\emph{On the structure of split Lie color algebras},
Linear Algebra Appl. \textbf{436} (2012), no.~2, 307--315.

\bibitem{CS2}
A.J. Calder\'on and J.M. S\'anchez,
\emph{On split Leibniz algebras},
Linear Algebra Appl. \textbf{436} (2012), no.~6, 1648--1660.

\bibitem{CS3}
M. Chaves and D. Singleton,
\emph{Phantom energy from graded algebras},
Modern Phys. Lett. A \textbf{22} (2007), no.~1, 29--40.

\bibitem{CFS}
R. Coquereaux, G. Esposito-Far\`ese and F. Scheck,
\emph{Noncommutative geometry and graded algebras in electroweak interactions},
Internat. J. Modern Phys. A \textbf{7} (1992), no.~26, 6555--6593.

\bibitem{D}
G. Dahl,
\emph{The doubly graded matrix cone and Ferrers matrices},
Linear Algebra Appl. \textbf{368} (2003), 171--190.

\bibitem{FKS}
\'E. Z. Fornaroli, M. Khrypchenko and E.A. Santulo,
\emph{Regular Hom-Lie structures on incidence algebras},
Rev. R. Acad. Cienc. Exactas Fis. Nat. Ser. A-Mat.
\textbf{117} (2023), article~122.


\bibitem{SWZ3}
B. Gai, C. Li, J. Sun, S. Wang and H. Zhu,
\emph{BiHom-Lie brackets and the Toda equation},
Symmetry \textbf{17} (2025), no.~12, article~2176.


\bibitem{GMM}
G. Graziani, A. Makhlouf, C. Menini and F. Panaite,
\emph{BiHom-associative algebras, BiHom-Lie algebras and BiHom-bialgebras},
SIGMA Symmetry Integrability Geom. Methods Appl. \textbf{11} (2015),
086, 34 pp.

\bibitem{GG1}
G. Greaves,
\emph{Cyclotomic matrices over the Eisenstein and Gaussian integers},
J. Algebra \textbf{372} (2012), 560--583.

\bibitem{GG2}
G. Greaves,
\emph{Small-span Hermitian matrices over quadratic integer rings},
Math. Comp. \textbf{84} (2015), no.~291, 409--424.

\bibitem{TG}
G. Greaves and G. Taylor,
\emph{Lehmer's conjecture for Hermitian matrices over the Eisenstein and
Gaussian integers},
Electron. J. Combin. \textbf{20} (2013), no.~1, Paper~42.

\bibitem{HLS}
J.T. Hartwig, D. Larsson and S.D. Silvestrov,
\emph{Deformations of Lie algebras using $\sigma$-derivations},
J. Algebra \textbf{295} (2006), no.~2, 314--361.

\bibitem{K}
M. Kochetov,
\emph{Gradings on finite-dimensional simple Lie algebras},
Acta Appl. Math. \textbf{108} (2009), no.~1, 101--127.

\bibitem{LS}
D. Larsson and S. D. Silvestrov,
\emph{Quasi-Hom-Lie algebras, central extensions and $2$-cocycle-like
identities},
J. Algebra \textbf{288} (2005), no.~2, 321--344.

\bibitem{LMM}
L. Liu, A. Makhlouf, C. Menini and F. Panaite,
\emph{BiHom-NS-algebras, twisted Rota--Baxter operators and generalized
Nijenhuis operators},
Results Math. \textbf{78} (2023), article~251.

\bibitem{MZ}
T. Ma and H. Zheng,
\emph{Some results on Rota--Baxter monoidal Hom-algebras},
Results Math. \textbf{72} (2017), 145--170.

\bibitem{PPS}
J. Patera, E. Pelantov\'a and M. Svobodov\'a,
\emph{Fine gradings of $\mathfrak{o}(4,\mathbb{C})$},
J. Math. Phys. \textbf{45} (2004), no.~6, 2188--2198.

\bibitem{PNM}
E. Peyghan, L. Nourmohammadifar and I. Mihai,
\emph{Para-Sasakian geometry on Hom-Lie groups},
Rev. R. Acad. Cienc. Exactas Fis. Nat. Ser. A-Mat.
\textbf{115} (2021), article~163.

\bibitem{SWZZ}
J. Sun, S. Wang, C. Zhang and H. Zhu,
\emph{Sweedler duality for Hom-$(\mathrm{co})$algebras and
Hom-$(\mathrm{co})$modules},
J. Nonlinear Math. Phys. \textbf{32} (2025), article~51.

\bibitem{V}
A. Verbovetsky,
\emph{Lagrangian formalism over graded algebras},
J. Geom. Phys. \textbf{18} (1996), 195--214.

\end{thebibliography}

\end{document}